\documentclass[a4paper]{amsart} 
\usepackage{pgf,tikz}
\usetikzlibrary{arrows}
\usepackage{mathrsfs}
\usepackage{amsthm}
\usepackage{amsmath}
\usepackage[foot]{amsaddr}
\usepackage{stmaryrd}
\usepackage{amssymb}
\usepackage{ulem}
\usepackage{geometry}
\newcommand{\comp}{\circ}
\newcommand{\id}{\text{id}}
\newcommand{\QR}[2]{
	   \left.\raisebox{1ex}{\ensuremath{#1}}~\slash~
	      \ensuremath{\mkern-1mu}\ensuremath{\mkern-1mu}
	         \raisebox{-1ex}{\ensuremath{#2}}\right.}

\newtheorem{ddef}{Definition}
\numberwithin{ddef}{section} % important bit
\newtheorem{prop}[ddef]{Proposition}

\newtheorem{lmm}[ddef]{Lemma}
\newtheorem{thm}[ddef]{Theorem}
\newtheorem{rmk}[ddef]{Remark}

\newtheorem{exa}[ddef]{Example}

\newtheorem*{rep@theorem}{\rep@title}
\newcommand{\newreptheorem}[2]{%
\newenvironment{rep#1}[1]{%
 \def\rep@title{#2 \ref{##1}}%
 \begin{rep@theorem}}%
 {\end{rep@theorem}}}

\newreptheorem{theorem}{Theorem}

\newreptheorem{lemma}{Lemma}

\title{Factorization Homology of Polynomial Algebras}
\author{Lennart D\"oppenschmitt}
\address{Department of Mathematics, University of Toronto, Canada}
\email{lennart@math.toronto.edu}
\begin{document}

\maketitle

\begin{abstract}
\noindent
We compute the factorization homology of a polynomial algebra over a 
compact and closed manifold with trivialized tangent bundle up to weak equivalence in a new way.
This calculation is based on the model of a graph complex and an explicit morphism 
into the codomain, which makes it possible to twist the algebra with a Maurer-Cartan 
element and potentially apply other deformations.
\end{abstract}

\section{Introduction}
\label{sec:Introduction}
\noindent
Factorization homology, also known as manifoldic homology, is an 
invariant of topological $n$-manifolds.  In this paper we present a new and more transparent method of its computation.
Based on the work of Beilinson and Drinfeld on factorization 
algebras the notion of factorization homology generalizes the topological 
chiral homology theory by Lurie with coefficients in $n$-disks algebras as well as 
labeled configuration spaces from Salvatore and Segal.
With origins in conformal field theory as well as the configuration space models of mapping spaces
it creates an overlap of various fields of interest and contributes for example to 
topological quantum field theory as an algebraic model for observables in \cite{cis} and others.
Applications can be found numerously in the literature. 
\\
Given an $\mathsf{e}_n$-algebra $A$, the factorization homology $\int_M A$ for a 
framed n-manifold $M$ may be defined by a colimit.
Unfortunately, this is notoriously hard to compute and usually avoided by employing methods of excision.
The necessity of other means of computation is evident by the large number of applications and affiliated concepts.
\\
Markarian has computed in \cite[Proposition 4]{nm} the 
factorization homology of a polynomial algebra using directly its definition as a colimit.  
The main result of this paper is an alternative 
computation of the factorization homology of the same polynomial algebra. 
It relies instead on a graph complex model introduced by Campos and Willwacher in \cite{cw}.
This model for the $n$-disks operad allows us to perform computations with graphs on a combinatorial level.
In choosing this accessible and easily manipulable approach, we 
enable further generalizations and modifications such as twistings of the coefficient 
algebra, which leads to the following improved result.
\begin{reptheorem}{thm:main2}
    Let $M$ be a compact and oriented manifold $M$ with trivialized tangend bundle.
    Let moreover $V$ be the shifted cotangent bundle $V = T^* [1-n] \mathbb{R}^N$ and consider the 
    the polynomial algebra $\mathcal{O} = \mathcal{O} \left( V \right)$.
    Then the factorization homology 
    $\int_M \mathcal{O} \llbracket \hbar \rrbracket$ of $M$
    with coefficients in the twisted polynomial algebra
    $\mathcal{O} \llbracket \hbar \rrbracket = \mathcal{O} ^{\hbar m_1 + \hbar^2 m_2 + \dotsm}$
    is weakly equivalent to the 
    algebra of twisted polynomials $S(H(M) \otimes V)$.
\end{reptheorem}
\noindent
The computation consists of rewriting the derived composition product with a cofibrant 
replacement of the graph complex model and the explicit construction of a map into the desired codomain.
Subsequently it is only left to show that this map is a weak equivalence using multiple nested spectral sequences.
In the course of setting up these spectral sequences, we make use of the Lambrecht- 
Stanley model and results from Sinha in \cite{ds}.

\section*{Acknowledgement}
\noindent
This paper is the result of a research project under the supervision of Thomas Willwacher.
He gave me with great enthusiasm and patience an invaluable insight in his research and helped me to create this paper.
I am very thankful for this opportunity as well as for his guidance and advice.

\section{Background}
\label{sec:Background}
\noindent
For the sake of introducing the proper framework for our computation and presenting 
notational conventions, we briefly recall the construction of the essential tools used in this paper.
\subsection{Operads}
\label{sec:Operads}
For the notational conventions of operads we will follow Fresse's book \cite{bf}.
Let $\left( \mathcal{C}, \otimes, \mathfrak{1} \right)$ be a symmetric monoidal category. A $\Sigma_*$-object in $\mathcal{C}$
is a collection $\left( C(0), C(1), \dotsc \right)$ of $\Sigma_n$-modules $C(n) \in \mathsf{Obj}(\mathcal{C})$ for $n \in \mathbb{N}$.
To a $\Sigma_*$-object we associate the endofunctor $S(C) \in \mathsf{End}(\mathcal{M})$
of a  symmetric monoidal category $\mathcal{M}$ over $\mathcal{C}$, where 
$S(C)$ applied to $M \in \mathsf{Obj}(\mathcal{M})$ is defined to be
\begin{align*}
    S(C,M) = \bigoplus _r ^{\infty} \left( C(r) \otimes M ^{\otimes r}  \right) _{\Sigma_r} 
\end{align*}
Denote by $\mathcal{C}^{\Sigma_*}$ the category of $\Sigma_*$-objects in $\mathcal{C}$ 
where a morphism $f: C \rightarrow D$ consists of a collection of
$\Sigma_*$-equivariant morphisms $f : C(n) \rightarrow D(n)$ in $\mathsf{Hom}(\mathcal{C})$.
Note that the map $S : \mathcal{C}^{\Sigma_*} \rightarrow \mathsf{End}(\mathcal{M})$ 
is functorial.
With the composition of $\Sigma_*$-objects $C$ and $D$ given by
\begin{align*}
    C \comp D = S(C,D)
\end{align*}
and the unit object
\begin{align*}
        \mathcal{I}(n) = \begin{cases} \mathfrak{1} \qquad n=0 \\ 0 \qquad \text{otherwise} \end{cases}
\end{align*}
the category $\left( \mathcal{C} ^{\Sigma_*}, \comp , \mathcal{I} \right) $ obtains 
the structure of a monoidal category over $\mathcal{C}$ and makes
$S : \left( \mathcal{C} ^{\Sigma_*}, \comp , \mathcal{I} \right) \rightarrow
\left( \mathsf{End}(\mathcal{M}), \comp , \text{Id} \right)$ into a monoidal functor.
\\
An operad in $\mathcal{C}$ comprises a $\Sigma_*$-object $\mathcal{P} \in \mathcal{C}^{\Sigma_*}$
together with a composition morphism $\mu : \mathcal{P} \comp \mathcal{P} \rightarrow \mathcal{P}$
and a unit morphism $\eta : I \rightarrow \mathcal{P}$, satisfying the usual monoidal unit and associativity condition.
The composition morphism of an operad can be described by morphisms
\begin{align*}
        \mu_r : \mathcal{P} (r) \otimes \mathcal{P} (n_1) \otimes \dotsm \otimes \mathcal{P} (n_r) \rightarrow \mathcal{P}(n_1 + \dotsm + n_r)
\end{align*}
or equivalently by partial composites
\begin{align*}
        \comp : \mathcal{P}(k) \otimes \mathcal{P}(l) \rightarrow \mathcal{P}(k+l-1)
\end{align*}
Let us furthermore use for an operad $\mathcal{P}$ the notation $\mathcal{P} ^{\vee}$ 
for its Koszul dual operad and $\mathcal{P}^*$ for its linear dual operad.

\subsection{Examples of Operads}
\label{sec:Examples of Operads}

Commonly used operads are the commutative operad $\mathsf{Com}$ and the Lie operad $\mathsf{Lie}$,
which carry the structure of the commutative product $\wedge$ and Lie bracket $\lbrace , \rbrace$ respectively.
The (framed) embedding spaces of $r$ disks $D^n$ into another disk $D^n$
\begin{align*}
        \mathsf{Disks}_n^{\text{(fr)}}(r) = \mathsf{Emb}^{\text{(fr)}}  (D^n \times \lbrace 1, \dotsc, r \rbrace, D^n )
\end{align*}
assemble with the composition of (framed) embeddings the (framed) little $n$-disks operad $\mathsf{Disks}_n^{\text{(fr)}}$ in the category of topological spaces.
The homology of the little $n$-disks operad inherits the operadic structure and is 
denoted by $\mathsf{e}_n = H _{\ast} (\mathsf{Disks}_n)$.
The $\mathsf{e}_n$ operad can be identified with the operadic composition
$\mathsf{Com} \comp \mathsf{Lie} \left[n-1\right] $ and is therefore generated by the commutative product
and the Lie bracket in arity two.
It's Koszul dual is $\mathsf{e}_n ^{\vee}  = \mathsf{e}^{\ast}_n \left[n\right]
= \mathsf{Com}^*\left[ n \right] \comp \mathsf{Lie} \left[ 1 \right]$.

\subsection{Algebras and Modules over Operads}
\label{sec:Algebras and Modules over Operads} 
An algebra over the operad $\mathcal{P}$, or short $\mathcal{P}$-algebra, is an object 
$A \in \mathsf{Obj}(\mathcal{C})$ with a morphism
$\lambda : S \left(\mathcal{P}, A \right) \rightarrow A$ such that the natural
unit, associativity and equivariance relations are satisfied.
This morphism can equivalently be expressed by morphisms
$\lambda_r : \mathcal{P}(r) \otimes A ^{\otimes r} \rightarrow A$.
Denote the category of $\mathcal{P}$-algebras by $\mathsf{alg} _{\mathcal{P}}$ where 
the morphisms are morphisms in $\mathsf{Hom}(\mathcal{C})$ that intertwine with the evaluation morphisms.
\\
A left $\mathcal{P}$-module $M$ is in a similar fashion an object
$M \in \mathsf{Obj}(\mathcal{C} ^{\Sigma_*})$ with an $\Sigma_*$-equivariant left $\mathcal{P}$-action
$\lambda : \mathcal{P} \comp M \xrightarrow{\lambda} M$ compatible with the operads 
associativity and unity relations. Left $\mathcal{P}$-modules together with morphisms 
$\mathsf{Hom}(\mathcal{C} ^{\Sigma_*})$ that intertwine with $\lambda$ form the 
category $\mathsf{mod}_{\mathcal{P}}$.\\
There is a functor
\begin{align*}
        \mathsf{alg} _{\mathcal{P} } \rightarrow \mathsf{mod} _{\mathcal{P}} 
\end{align*}
given by $A \mapsto A ^{\otimes} = \left( 0, A, 0, 0, \dotsc \right)$
\\
Right $\mathcal{P}$-modules are constructed in the same way with a right $\mathcal{P}$-
action, but note that due to the lack of symmetry in the operadic composition 
left and right $\mathcal{P}$-modules are entirely different objects.
The composition product $M \comp _{\mathcal{P}} N$ of a right $\mathcal{P}$-module $M$ 
and a left $\mathcal{P}$-module $N$ is defined through the pushout diagram
\begin{center}
    \begin{tikzpicture}[scale=1]
        \node (00) at (0,0) {$M \comp N$};
        \node (01) at (0,2) {$M \comp \mathcal{P} \comp N$};
        \node (10) at (3,0) {$M \comp _{\mathcal{P}} N$};
        \node (11) at (3,2) {$M \comp N$};
        \draw[->, black,font=\normalsize]
        (01) edge node[above] {$$} (11)
        (00) edge node[above] {$$} (10)
        (01) edge node[left] {$$} (00)
        (11) edge node[right] {$$} (10);
        \draw[black] (2,.5) -- (2.5,.5) -- (2.5,1);
    \end{tikzpicture}
    \end{center}
When taking the composition product of a right $\mathcal{P}$-module and a $\mathcal{P}$-algebra
$A$, we associate to $A$ its left $\mathcal{P}$-module to use the composition product of $\mathcal{P}$-modules.
The derived composition product $M ~\widehat \comp _{\mathcal{P} }  ~ A := \widehat M \comp _{\mathcal{P} } A$
is the composition product of a cofibrant replacement $\widehat M$ with A.
For example, given a (framed) manifold $M$ the spaces
$\mathsf{Disks}_M^{\text{(fr)}}(r) = \mathsf{Emb^{(fr)}} (D^n \times \lbrace 1, \dotsc, r \rbrace, M )$
from a right $\mathsf{Disks}_n^{\text{(fr)}}$-module, where the action of $\mathsf{Disks}_n^{\text{(fr)}}$ is given by a (framed) embedding.

\subsection{The Graph Complex $\mathsf{Graphs}_M$}
\label{sec:Graph Complex}

In their paper \cite{cw} Campos and Willwacher build $\mathsf{Graphs}_M$ as a model for 
the configuration space of a sufficiently well behaved manifold, which we will denote from here on by $M$.
That means that $M$ has to be compact, closed and its tangent bundle $TM$ is trivialized.
The construction of this model is summarized in the following.
\\
Define for $r \geq 0$ and $n \in \mathbb{N}$  the space $\mathsf{Gra}_n(r)$ to be the quotient of the free 
graded commutative algebra generated by elements $s_{ij}$ of degree $n-1$, by the 
relations $s _{ij} = (-1)^n s _{ji}$, where $1 \leq i \neq j \leq r$.
Elements $\Gamma \in \mathsf{Gra}_n(r)$ can be pictured as the linear combination of graphs with $r$ vertices.
The spaces $\mathsf{Gra}_n(r)$ for $r \in \mathbb{N}$ form the 
operad $\mathsf{Gra}_n$ for which the partial operadic composition $\comp_i $ of two graphs is declared to be 
the insertion of the second graph into the $i ^{th}$ vertex of the first graph with summation over all possibilities of 
reconnecting the edges.
\begin{align*}
        \mathsf{Gra}_n (k) \comp _i \mathsf{Gra}_n (l) \rightarrow \mathsf{Gra}_n(k + l - 1)
\end{align*}
The operad $\mathsf{Graphs}_n$ consists of graphs with additional indistinguishable vertices of degree $n$ which are not labeled.
In contrast to the external vertices that are labeled by $1, \dotsc, r$, we call these internal vertices.
Adding internal vertices is made precise by Willwacher under the term of operadic twisting.
As the third and final step we define $\mathsf{Graphs}_M$, as a right dg-module over
$\mathsf{Graphs}_n$. The only difference to $\mathsf{Graphs}_n$ is that every
vertex can be decorated by a cohomology class $\alpha \in H^{\ast} (M)$.

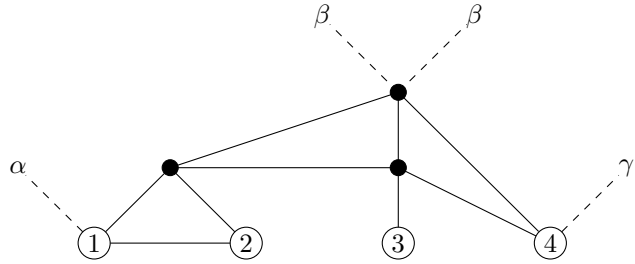
\begin{center}
\begin{tikzpicture}[scale=1]
    \clip(-1,0) rectangle (10,5);
    \node (v1) at (1,1) {$1$}; 
    \node (v2) at (3,1) {$2$}; 
    \node (v3) at (5,1) {$3$}; 
    \node (v4) at (7,1) {$4$}; 
    \node (i1) at (2,2) {$\bullet$}; 
    \node (i2) at (5,2) {$\bullet$}; 
    \node (i3) at (5,3) {$\bullet$}; 
    \node (d1) at (0,2) {$\alpha$}; 
    \node (d2) at (4,4) {$\beta$}; 
    \node (d3) at (6,4) {$\beta$}; 
    \node (d4) at (8,2) {$\gamma$}; 
    \draw
    (v1) circle (6pt)
    (v2) circle (6pt)
    (v3) circle (6pt)
    (v4) circle (6pt);
    \draw[fill=black]
    (i1) circle (3pt)
    (i2) circle (3pt)
    (i3) circle (3pt);
    \draw[-, black, font=\normalsize]
    (v1) -- (v2)
    (1.15,1.15) -- (2,2)
    (2.85,1.15) -- (2,2)
    (2,2) -- (5,2)
    (2,2) -- (5,3)
    (5,2) -- (5,3)
    (5,1.21) -- (5,2)
    (6.813,1.094) -- (5,2)
    (6.846,1.154) -- (5,3);
    \draw[-, dashed, black, font=\normalsize]
    (0.85,1.15) -- (0.1,1.9)
    (5,3) edge (4.10,3.90)
    (5,3) edge (5.9,3.9)
    (7.15,1.15) edge (7.9,1.9);
\end{tikzpicture}
\end{center}
\noindent
The degree of a graph $\Gamma \in \mathsf{Graphs} _M$ is
\begin{align}
\begin{split}
    \label{eq:graphdeg}
        \mathsf{deg}(\Gamma) = n \cdot s - \left( n-1 \right) \cdot e + \sum _{\alpha} \mathsf{deg}_{H^{\ast}} (\alpha)
\end{split}
\end{align}
where $s$ is the number of internal vertices and $e$ the number of edges.
Due to the symmetry properties of edges and vertices we may define a notion of 
orientation for the cases of $n$ odd and even. In the case that $n$ is odd the orientation
is given by the order of external vertices, a direction of every edge and an order of all odd 
decorations. In the case that $n$ is even the orientation is given by an order of edges and an 
order of all odd decorations.
One term of the differential on this module is $\delta_{\text{split}}$, which splits an internal vertex out of both internal and 
external vertices and sums over all possible ways of reconnecting the incident edges.
\begin{align*}
    \delta _{\text{split}} \Gamma = (-1) ^{\vert \Gamma \vert} \sum _{v \in \Gamma}
    \Gamma \circ_v \text{$\circ$-$\bullet$}
\end{align*}
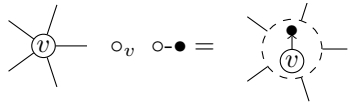
\begin{align*}
        \raisebox{-4.7ex}{
        \begin{tikzpicture}[scale=0.4]
            \clip(0,0) rectangle (4,4);
            \node (v1) at (2,2) {$v$}; 
            \draw (v1) circle (11pt);
            \draw (2.40,2.00) -- (3.43, 2.00);
            \draw (2.12,2.38) -- (2.44, 3.36);
            \draw (1.68,2.23) -- (0.85, 2.84);
            \draw (1.68,1.77) -- (0.85, 1.16);
            \draw (2.12,1.62) -- (2.44, 0.64);
        \end{tikzpicture}}
        \circ_v ~\text{$\circ$-$\bullet$} =
        \raisebox{-6.3ex}{
        \begin{tikzpicture}[scale=0.4]
            \clip(0,0) rectangle (4,4);
            \node (v1) at (2,2.1) {$v$}; 
            \node (i1) at (2,3.1) {$\bullet$}; 
            \draw (v1) circle (11pt);
            \draw[dashed] (2,2.5) circle (28pt);
            \draw[->, black, font=\normalsize] (2,2.5) -> (2,3.0);
            \draw (2.00 + 1.00, 2.50 + 0.00) -- (2.00 + 1.82, 2.50 + 0.00);
            \draw (2.00 + 0.31, 2.50 + 0.95) -- (2.00 + 0.56, 2.50 + 1.73);
            \draw (2.00 - 0.81, 2.50 + 0.59) -- (2.00 - 1.47, 2.50 + 1.07);
            \draw (2.00 - 0.81, 2.50 - 0.59) -- (2.00 - 1.47, 2.50 - 1.07);
            \draw (2.00 + 0.31, 2.50 - 0.95) -- (2.00 + 0.56, 2.50 - 1.73);
        \end{tikzpicture}}
\end{align*}
The other term $\delta _{\text{pair}}$ pairs any two complementary decorations $\alpha$ and $\beta$
as $\langle \alpha, \beta \rangle = (-1) ^{\vert \alpha \vert} \int_M \alpha \wedge \beta$
and connects the vertices to which they are attached to by an edge.
The label or direction of the new edge is subject to conventions.
\begin{align*}
        \delta _{\text{pair}} \left(
        \raisebox{-4.7ex}{
        \begin{tikzpicture}[scale=0.4]
            \clip(0,0) rectangle (3,4);
            \node (e1) at (.5,3) {$1$}; 
            \node (e2) at (.5,1) {$2$}; 
            \node (d1) at (2.5,3.5) {$\alpha^1$}; 
            \node (d2) at (2.5,0.5) {$\beta^2$}; 
            \draw
            (e1) circle (11pt)
            (e2) circle (11pt);
            \draw[-, dashed, black, font=\normalsize]
            (e1) edge (d1)
            (e2) edge (d2);
        \end{tikzpicture}}
        \right) = 
        \raisebox{-4.7ex}{
        \begin{tikzpicture}[scale=0.4]
            \clip(0,0) rectangle (3.5,4);
            \node (f) at (1,2) {$\langle \alpha, \beta \rangle$};
            \node (e1) at (3,3) {$1$}; 
            \node (e2) at (3,1) {$2$}; 
            \draw
            (e1) circle (11pt)
            (e2) circle (11pt);
            \draw[->, black, font=\normalsize]
            (3,2.6) -> (3,1.4);
        \end{tikzpicture}}
\end{align*}
Denote the sum of these two differentials by $\delta = \delta _{\text{split}} + \delta _{\text{pair}}$.
The union of two graphs $\Gamma$ and $\Sigma$ with the orietation given by 
concatinating the order of vertices or edges from left to right is denoted
by $\Gamma \cdot \Sigma$.
The failure of $\delta _{\text{pair}}$ to be a derivation with respect to this product is measured by the Lie bracket 
\begin{align*}
    \lbrace \Gamma, \Sigma \rbrace_G = \delta_{\text{pair}} ( \Gamma \cdot \Sigma )
    - \delta_{\text{pair}} (\Gamma) \cdot \Sigma - (-1) ^{\vert \Gamma \vert} \Gamma \cdot \delta_{\text{pair}} (\Sigma)
\end{align*}
In the literature a linear combination of possibly decorated graphs which only contain internal vertices
is of interest. This so called partition function $\mathsf{z}_M$ satisfies the Maurer-Cartan equation in $\mathsf{Graphs}_M$
\begin{align*}
    \delta \mathsf{z}_M + \frac{1}{2} \lbrace \mathsf{z}_M, \mathsf{z}_M \rbrace_G = 0
\end{align*}
Such an element $\mathsf{z}_M$ then defines a twisting differential $\delta ^{\mathsf{z}_M} = \lbrace \mathsf{z}_M, ~ \rbrace_G $.
The full differential on $\mathsf{Graphs}_M$ is $\delta + \delta ^{\mathsf{z}_M}$.
\subsection{The Polynomial Algebra $\mathcal{O}$}
\label{sec:Polynomial Algebra}
Consider for two fixed positive integers $n$ and $N$ the the polynomial algebra $\mathcal{O}$  
over the shifted cotangent bundle $V = T^* [1-n] \mathbb{R}^N$ of Euclidean space.
\begin{align*}
        \mathcal{O} = \mathcal{O} \left(T^* [1-n] \mathbb{R}^N  \right)
\end{align*}
Given a basis $x_1, \dotsc, x_N$ of $\mathbb{R} ^N$, we may write $\mathcal{O}$ as the free 
commutative algebra in coordinates $p_1, \dotsc, p_N$ of degree $\left(1-n\right) $ and.
$x_1, \dotsc, x_N$ of degree $0$, also denoted by $y_1, \dotsc, y_{2N}$ for convenient notation.
The graded Poisson bracket on $\mathcal{O}$ is uniquely defined by the relations
\begin{align*}
        \lbrace p_i, x_j \rbrace &= \delta _{ij} \\
        \lbrace p_i, p_j \rbrace &= \lbrace x_i, x_j \rbrace = 0
\end{align*}
and the symmetry property $\lbrace f,g \rbrace = (-1) ^{n + \vert f \vert \vert g \vert} \lbrace g,f \rbrace$.
\subsection{Twisting}
\label{sub:Twisting}
$\left( \mathcal{O}, \lbrace ~,~ \rbrace \right)$ is in particular a dg Lie algebra with the zero differential.
Therefore the Maurer-Cartan equation for $m \in \mathcal{O}$ simplifies to $\lbrace m, m \rbrace = 0$.
Any such $m$ satisfying the Maurer-Cartan equation gives rise to a twisting differential 
\begin{align*}
d^m = - \lbrace m, ~ \rbrace
\end{align*}
which is a derivation with respect to the product in $\mathcal{O}$.
The twisted polynomial algebra is given by $\mathcal{O} ^m = \left( \mathcal{O}, \lbrace ~, ~ \rbrace, d^m \right)$.
For the sake of simplicity, let us define as a first step
$\mathcal{O} \llbracket \hbar \rrbracket = \mathcal{O}^m$ for $m = \hbar m_1 + \hbar^2 
m_2 + \dotsm$ with $\hbar > 0$ and $m_i \in \text{MC}(\mathcal{O} )$.
\subsection{Decorated Polynomials}
\label{sub:Decorated Polynomials}
A slight modification of the algebra $\mathcal{O}$ introduced above is
\begin{align*}
        \mathcal{S} = S(H(M) \otimes V)
\end{align*}
the polynomial algebra in formal variables $p_1, \dotsc, p_N$ and $x_1, \dotsc, x_N$, 
each one tensored with a cohomology class $\alpha \in H^{\ast} (M)$.
To make the distinction to polynomials in $\mathcal{O}$ as apparent as possible, we will denote polynomials $F \in \mathcal{S}$ by capital letters.
The degree of a decorated polynomial $F$ is iteratively defined by the degree of its variables
\begin{align*}
    \mathsf{deg} (\alpha \otimes y_k) = \mathsf{deg} (\alpha) + \mathsf{deg} (y_k)
\end{align*}
Polynomials $F,G \in \mathcal{S}$ satisfy the graded commutativity relation
\begin{align*}
    F \cdot G = (-1)^{\vert F \vert \vert G \vert} G \cdot F 
\end{align*}
We consider the differential $\Delta$ on $\mathcal{S}$ that acts in the following way on decorated polynomials.
\begin{align*}
        \Delta(F) = \sum_{\substack{j = 1, \dotsc, N \\ \omega \text{ basis of } H^{\ast}}}
        \Bigg( \frac{\partial}{\partial \left( \omega \otimes p_j \right) } 
        \frac{\partial}{\partial \left( \omega^* \otimes x_j \right) } 
        \Bigg) \Big( F \Big)
\end{align*}
The Poincar\'e duality pairing $\langle \alpha, \beta \rangle$ is non degenerate, because $M$ is compact and oriented.
With this and the Poisson bracket on $\mathcal{O}$ we obtain a Poisson bracket on $\mathcal{S}$.
\begin{align*}
        \lbrace \alpha \otimes y_k, \beta \otimes y_l \rbrace _{\mathcal{S}} = \langle \alpha, \beta \rangle \lbrace y_k, y_l \rbrace 
\end{align*}
for $y_k, y_l \in \lbrace p_1, \dotsc, p_N, x_1, \dotsc, x_N \rbrace$.
This bracket may be expressed as
\begin{align*}
    \lbrace F, G \rbrace _{\mathcal{S}} 
    = \sum_{\substack{j = 1, \dotsc, N \\ \omega \text{ basis of } H^{\ast}}}
    \frac{\partial F}{\partial \left( \omega \otimes p_j \right) } 
    \frac{\partial G}{\partial \left( \omega ^* \otimes x_j \right) } 
    + (-1)^n
    \frac{\partial G}{\partial \left( \omega \otimes p_j \right) } 
    \frac{\partial F}{\partial \left( \omega ^* \otimes x_j \right) } 
\end{align*}

\subsection{Factorization Homology}
\label{sec:Factorization Homology}
For a $\mathsf{Disks}_n^{\text{fr}}$-algebra $A$, we define its algebraic factorization homology over
a manifold $M$ with trivialized tangent bundle to be the derived composition product
\begin{align*}
    \int_M A := C_* \left( \mathsf{Disk}_M \right) ~ \widehat \comp _{\mathsf{e}_n} ~ A
\end{align*}
For compact and closed manifolds Campos and Willwacher have shown that
$\mathsf{Graphs}_M$ is a model for chains $C_* \left( \mathsf{Disks}_M \right) $.
That is, the pair of right module and operad $\left( \mathsf{Graphs}_M \circlearrowleft \mathsf{e}_n \right)$ is weakly equivalent to the pair
$\left( C_* \left( \mathsf{Disks}_M  \right) \circlearrowleft \mathsf{e}_n \right)$. 
We can therefore rewrite for an $\mathsf{e}_n$-algebra $A$ its factorization homology as
\begin{align*}
        \int_M A \simeq \mathsf{Graphs}_M ~\widehat \comp _{\mathsf{e}_n} A
\end{align*}
To compute the derived composition product we choose a cofibrant replacement of $\mathsf{
    Graphs}_M$ using the bar-cobar construction  of right $\mathsf{e}_n$ modules\begin{align*}
        \widehat{\mathsf{Graphs}_M}
        &= \mathsf{Free}_{\mathsf{e}_n}\left(\mathsf{coFree}_{\mathsf{e}_n^{\vee}} \left( \mathsf{Graphs} _M \right) \right) \\
        &= \mathsf{Graphs}_M \comp \mathsf{e}^{\vee}_n \comp \mathsf{e}_n
\end{align*}

\section{Computation of Factorization Homology}
\label{sec:Compution of Factorization Homology}
\noindent
Now that the framework is set, we can prove the main theorem, which computes the 
factorization homology of the polynomial algebra introduced above.
\begin{thm}
    \label{thm:main1}
    Let $M$ be a compact and oriented manifold $M$ with trivialized tangend bundle.
    Let moreover $V$ be the shifted cotangent bundle $V = T^* [1-n] \mathbb{R}^N$ and consider the 
    the polynomial algebra $\mathcal{O} = \mathcal{O} \left( V \right)$.
    Then the factorization homology 
    $\int_M \mathcal{O}$ of $M$
    with coefficients in the polynomial algebra $\mathcal{O}$
    is weakly equivalent to the algebra of twisted polynomials $S(H(M) \otimes V)$.
\end{thm}
\subsection{Map Construction}
\label{sub:Map Construction}
To perform the calculation, we start by rewriting the derived composition product
with a cofibrant replacement of $\mathsf{Graphs}_M$.
\begin{align*}
        \int_M \mathcal{O}
        &= \widehat{\mathsf{Graphs} _M} ~\comp _{\mathsf{e}_n} \mathcal{O} \\
        &= \mathsf{Free}_{\mathsf{e}_n}\left(\mathsf{coFree}_{\mathsf{e}_n^{\vee}}
            \left( \mathsf{Graphs} _M \right) \right) ~\comp _{\mathsf{e}_n} \mathcal{O} \\
        &= \mathsf{coFree}_{\mathsf{e}_n^{\vee}} \left( \mathsf{Graphs} _M \right) ~\comp \mathcal{O} \\
        &= \mathsf{Graphs} _M ~\comp \mathsf{e}_n^{\vee} \comp \mathcal{O} \\
        &= \mathsf{Graphs} _M \comp \mathsf{Com}^* \left[ n \right]  \comp \mathsf{Lie}^*\left[ 1 \right]  \comp \mathsf{Com} \comp V
\end{align*}
We denote this complex by $\mathcal{D}$ as this will be the domain of the map we are constructing in this section.
It is equipped with the differential
\begin{align*}
    \partial = \delta _{\text{split}} + \delta _{\text{pair}} + \delta^{\mathsf{z}_M} + d ^{\mathsf{com}} _{\mathsf{mod}} + d ^{\mathsf{com}} _{\mathsf{alg}} + d ^{\mathsf{lie}} _{\mathsf{mod}} + d ^{\mathsf{lie}} _{\mathsf{alg}} 
\end{align*}
The differentials of the form $d ^{\mathsf{A}} _{\mathsf{B}}$ are induced by twisting 
morphisms from the (co)operad $\mathsf{A}$ into the algebra or module specified by $\mathsf{B}$.
The action of $d ^{\mathsf{com}} _{\mathsf{mod}}$ is explained in more depth in the proof of lemma (\ref{lmm:isomod}).
$\mathcal{D}$ inherits a Lie bracket from $\mathsf{Graphs}_M$
\begin{align*}
    \lbrace \Gamma \otimes f_I, \Sigma \otimes f_J \rbrace _{\mathcal{D}}  = \lbrace \Gamma, \Sigma \rbrace_{G} \otimes f _{IJ} 
\end{align*}
With the composition of operads, we can moreover rewrite the codomain as $\mathcal{S} = \mathsf{Com} \comp H^*(M) \comp V$.
Let us first consider the maps
\begin{align*}
    \varphi_r' \quad : \quad \mathsf{Graphs}_M(r) \otimes \mathcal{O} ^{\otimes r} &\rightarrow \mathcal{S} \\
    \Gamma \otimes f_1 \otimes \cdots \otimes  f_r 
    &\mapsto \mu _{\text{mult}} \Bigg( \prod _{\substack{(i,j) \\ (i, \alpha)}} \Xi _{ij} \Big(f_1 \otimes \cdots \otimes f_r\Big) \Bigg)
\end{align*}
where every edge $(i,j)$ and decoration $(i,\alpha)$ is associated to an operator $\Xi _{ij}$
or $\Xi _{i\alpha}$ respectively, acting on the $i ^{th}$ and/or $j ^{th}$ factor in the tensor product.
\begin{align*}
    \Xi _{ij} &= \sum^{2N}_{k=1} \lbrace y_k, \tilde y_k \rbrace  \left( \frac{\partial}{\partial y_k} \right)_{(i)}
        \left( \frac{\partial}{\partial \tilde y_k} \right)_{(j)} \qquad \qquad
    \Xi _{i \alpha} = \sum^{2N}_{k=1} \left( \alpha \otimes y_k \right)
        \left( \frac{\partial}{\partial \tilde y_k} \right)_{(i)}
\end{align*}
Any graph that contains one or more internal vetrices gets send to zero by this map.
Subsequently $\mu _{\text{mult}} : f_1 \otimes \dotsm \otimes f_k \mapsto f_1 \cdot \dotsm
\cdot f_k$ replaces the tensor product with the product in $\mathcal{S}$. Combined this gives the map
\begin{align*}
        \varphi' : \bigoplus _r \mathsf{Graphs} _M(r) \otimes \mathcal{O}^{\otimes r}  \rightarrow S \left( H(M) \otimes V \right)
\end{align*}
\begin{exa}
    These definitions should become clearer by looking at an example of a graph in arity 
    three acting on functions $f,g,h \in \mathcal{O}$ in the case than $n$ is odd.
\begin{align*}
    \varphi'
    \left( 
        \raisebox{-4.7ex}{
        \begin{tikzpicture}[scale=0.4]
        \clip(0,0) rectangle (4,4);
        \node (e1) at (2.5,3) {$1$}; 
        \node (e2) at (1.5,1) {$2$}; 
        \node (e3) at (3.5,1) {$3$}; 
        \node (d1) at (0.5,3) {$\alpha$}; 
        \draw
        (e1) circle (11pt)
        (e2) circle (11pt)
        (e3) circle (11pt);
        \draw[->, black, font=\normalsize]
        (2.3,2.66) -- (1.68,1.34);
        \draw[->, black, font=\normalsize]
        (1.9,1.0) -- (3.1,1.0);
        \draw[-, dotted, thick, black, font=\normalsize]
        (e1) edge (0.9,3.0);
        \end{tikzpicture}}
        \otimes f \otimes g \otimes h
    \right)
    = \sum _{j,k,l = 1} ^{2N} \lbrace y_k, \tilde y_k \rbrace 
    \lbrace y_l, \tilde y_l \rbrace 
    \left( \alpha \otimes y_j \right) \frac{\partial^2 f}{\partial \tilde y_j \partial y_k}
    \frac{\partial^2 g}{\partial \tilde y_k \partial y_l} \frac{\partial h}{\partial \tilde y_l} 
\end{align*}
\end{exa}
\noindent
Next, we need to check that this map passes on to the composition of
$\mathsf{Graphs} _M$ and $\mathcal{O}$ as $\mathsf{e}_n$-modules.
For this it is necessary for the action of the generators of $\mathsf{e}_n$ 
on the module and the algebra to coincide when $\varphi'$ is applied.
The commutative product and Lie bracket include into $\mathsf{Graphs}_M$
as $( \circ\text{ }\circ )$ and $( \circ\text{-}\circ )$ respectively.
\begin{rmk}
\label{rmk:act}
    In the following, calculations edges and decorations exhibit the same behavior, it is 
    therefore possible to reduce these calculations without loss of generality to non decorated graphs.
    Starting with the simplest case of only one edge in a graph, say between the vertices $1$ and $2$, the product rule implies
    \begin{align*}
        & ~~~ \Xi _{12} \left( f_a f_b  \otimes f_2 \otimes \dotsm \otimes f _{r}  \right) \\
        &= \sum^{N}_{k=1} \lbrace y_k, \tilde y_k \rbrace ~ \frac{\partial (f_a f_b)}{\partial y_k} 
        \otimes \frac{\partial f_2}{\partial \tilde y_k}\otimes \dotsm \otimes f _{r} \\
        &= \sum^{2N}_{k=1} \lbrace y_k, \tilde y_k \rbrace ~ \left( f_a \cdot \frac{\partial f_b}{\partial y_k} 
        + \frac{\partial f_a}{\partial y_k} \cdot f_b \right)
        \otimes \frac{\partial f_2}{\partial \tilde y_k}\otimes \dotsm \otimes f _{r} \\
        &= \left( \Xi _{aj} + \Xi _{bj} \right) \left( f_a \otimes f_b  \otimes f_2 \otimes \dotsm \otimes f _{r}  \right) \\
    \end{align*}
For an index set $J = \lbrace j_1, \dotsc, j_k \rbrace$ we write $I^C$ for the
compliment of an index subset $I \subset J$.
Abbreviate moreover for such an index set the products $\Xi _{i J} = \Xi _{i j_1} \dotsm \Xi _{i j_k}$ and
$f_J = f _{j_1} \otimes \dotsm \otimes f _{j_k}$.
This amounts to the more general case where we take into account multiple edges at one vertex.
\begin{align*}
        \Xi _{1 J} \left( f_a f_b  \otimes f_2 \otimes \dotsm \otimes f _{r}  \right)
        = \sum _{I \subset J}\Xi _{a I} \Xi _{b I^C}\left( f_a \otimes f_b  \otimes f_2 \otimes \dotsm \otimes f _{r}  \right)
\end{align*}
Take $J$ to be the index set of all vertices connected to $1$. Finally this leads to
\begin{align*}
    & \varphi' \Big( \Gamma \otimes f_a f_b \otimes \cdots \otimes f_{r} \Big) \\
    &= \mu _{\text{mult}} \Bigg( \prod _{\substack{(i,j)}} \Xi _{ij}
    \Big(f_a f_b \otimes \cdots \otimes f_{r}\Big) \Bigg) \\
    &= \mu _{\text{mult}} \Bigg( \Xi _{1 J} \prod _{\substack{(i,j) \\ i,j \neq 1}} \Xi _{ij}
    \Big(f_a f_b \otimes \cdots \otimes f_{r}\Big) \Bigg) \\
    &=  \mu _{\text{mult}} \Bigg( \sum _{I \in \mathcal{P} (A_1)} 
    \Xi _{1 I} \Xi _{2 I^C}  \prod _{\substack{\text{edge} \\ (i,j) \\ i,j \neq 1}} \Xi _{ij}
    \Big(f_a \otimes f_b \otimes \cdots \otimes f_{r}\Big) \Bigg) \\
    &= \varphi' \Big( \left( \Gamma \comp_1 \big(\text{$\circ$ $\circ$}\big) \right) \otimes f_a \otimes f_b \otimes \cdots \otimes f_{r} \Big) \\
\end{align*}
The calculation for the generator $\lbrace ~,~ \rbrace$ is very similar,
since we can write $\lbrace f_a, f_b \rbrace = \mu _{mult} \left( \Xi _{ab} \left( f_a \otimes f_b \right) \right)$.
This additional factor $\Xi _{ab}$ in the calculation above corresponds then to an additional edge between the vertices which are being inserted in $\Gamma$.
Therefore we need to compose $\Gamma$ with
$(\circ$-$\circ)$ instead of $\left(\circ~\circ\right)$.
This means we have
\begin{align*}
        \varphi' \Big( \Gamma \otimes \lbrace f_a, f_b \rbrace \otimes f_2 \otimes \cdots \otimes f_{r} \Big)
        &= \varphi' \Big( \left( \Gamma \comp_1 \big(\text{$\circ$-$\circ$}\big) \right) \otimes f_a \otimes f_b \otimes \cdots \otimes f_{r} \Big) \\
\end{align*}
\end{rmk}
\noindent
Hence in the following diagram the induced quotient map $\widetilde \varphi$ is well defined and we can establish the desired 
morphism $\varphi = \widetilde \varphi \circ \pi$ as a composition.
\begin{center}
\begin{tikzpicture}[scale=1]
\label{fig:phi}
    \node (gd) at (0,4) {$\mathsf{Graphs} _M ~\widehat \comp _{\mathsf{e}_n} \mathcal{O}$};
    \node (gn) at (0,2) {$\mathsf{Graphs} _M \comp _{\mathsf{e}_n} \mathcal{O}$};
    \node (gnr) at (1.3,2) {$$};
    \node (gm) at (0,0) {$\bigoplus _r \mathsf{Graphs} _M(r) \otimes \mathcal{O} ^{\otimes r} $};
    \node (gmr) at (1.8,0.0) {$$};
    \node (s) at (6,4) {$S( H(M) \otimes V)$};
    \node (sr) at (6.3,3.8) {$$};
    \draw[->, black, font=\normalsize]
    (gd) edge node[left] {$\pi$} (gn)
    (gnr) edge node[below] {$\widetilde \varphi$} (s)
    (gmr) edge node[below] {$\varphi'$} (sr)
    (gm) edge node[left] {$$} (gn);
    \draw[->, dashed, black, font=\normalsize]
    (gd) edge node[above ] {$\varphi$} (s);
\end{tikzpicture}
\end{center}
For the sake of simplicity we will abbriviate
$\varphi \left( \Gamma \otimes f_I \right) = \Gamma \comp f_I$
\begin{prop}
    \label{prop:dcptl}
    The map $\varphi : \left( \mathcal{D}, \partial, \lbrace , \rbrace_{\mathcal{D}} \right)
    \rightarrow \left( \mathcal{S}, \Delta, \lbrace , \rbrace_{\mathcal{S}}  \right)$ is
    a morphism of dg Lie modules over $\mathsf{e}_n$.  More specifically,
    $\varphi \delta _{\text{pair}} = \Delta \varphi$ 
    whereas all other terms of $\partial$ vanish on $\mathcal{S}$.
\end{prop}
\begin{proof}
    Due to the commutativity of the partial derivatives it is sufficient to verify the claim on the following simple case.
    We are using here that $\Delta$ applied to monomials in $\mathcal{S}$ vanishes.
    \begin{align*}
        &~~ \left( \Delta \varphi \right) \left( 
                        \raisebox{-4.7ex}{
                        \begin{tikzpicture}[scale=0.4]
                        \clip(0,0) rectangle (3,4);
                        \node (e1) at (.5,3) {$1$}; 
                        \node (e2) at (.5,1) {$2$}; 
                        \node (d1) at (2.5,3.5) {$\alpha$}; 
                        \node (d2) at (2.5,0.5) {$\beta$}; 
                        \draw
                        (e1) circle (11pt)
                        (e2) circle (11pt);
                        \draw[-, dashed, black, font=\normalsize]
                        (e1) edge (d1)
                        (e2) edge (d2);
                        \end{tikzpicture}}
                        \otimes f \otimes g \right) \\
        &= \Delta \left( \sum _{k, l} (\alpha \otimes y_k)
            \frac{\partial f}{\partial \tilde y_k} ( \beta \otimes y_l) \frac{\partial g}{\partial \tilde y_l}\right) \\
        &= \sum _{k, l}
        \lbrace \alpha \otimes y_k,  \beta \otimes y_l \rbrace _{\mathcal{S}} 
        \frac{\partial f}{\partial \tilde y_k}  \frac{\partial g}{\partial \tilde y_l} \\
        &= \sum _{k, l} \langle \alpha, \beta \rangle
        \lbrace y_l, y_k \rbrace
        \frac{\partial f}{\partial \tilde y_k} \frac{\partial g}{\partial \tilde y_l} \\
        &= \sum _{l }\langle \alpha, \beta \rangle
        \lbrace y_l, \tilde y_l \rbrace
        \frac{\partial f}{\partial y_l} \frac{\partial g}{\partial \tilde y_l} \\
        &=  ~~ \varphi \left( 
                        \raisebox{-4.7ex}{
                        \begin{tikzpicture}[scale=0.4]
                        \clip(0,0) rectangle (3.5,4);
                        \node (f) at (1,2) {$\langle \alpha, \beta \rangle$};
                        \node (e1) at (3,3) {$1$}; 
                        \node (e2) at (3,1) {$2$}; 
                        \draw
                        (e1) circle (11pt)
                        (e2) circle (11pt);
                        \draw[->, black, font=\normalsize]
                        (3,2.6) -- (3,1.4);
                        \end{tikzpicture}}
                        \otimes f \otimes g \right) \\
        &= \varphi \left( \delta _{\text{pair}} \left( 
                        \raisebox{-4.7ex}{
                        \begin{tikzpicture}[scale=0.4]
                        \clip(0,0) rectangle (3,4);
                        \node (e1) at (.5,3) {$1$}; 
                        \node (e2) at (.5,1) {$2$}; 
                        \node (d1) at (2.5,3.5) {$\alpha$}; 
                        \node (d2) at (2.5,0.5) {$\beta$}; 
                        \draw
                        (e1) circle (11pt)
                        (e2) circle (11pt);
                        \draw[-, dashed, black, font=\normalsize]
                        (e1) edge (d1)
                        (e2) edge (d2);
                        \end{tikzpicture}}\right)
                        \otimes f \otimes g \right) \\
    \end{align*}
    As $\varphi$ sends graphs containing any internal vertices to zero and both of $\delta_{\text{split}}$ and
    $\delta ^{\mathsf{z}_M}$ result in a graph with internal vertices, these differentials correspond to zero.
    The compatibility of $\varphi$ with the brackets follows then immediately from their 
    relation to $\delta_{\text{pair}}$ and $\Delta$.
    We may write this property as
    \begin{align*}
        \lbrace \Gamma \comp f_I, \Sigma \comp f_J \rbrace _{\mathcal{D}} = \lbrace \Gamma, \Sigma \rbrace _{G} \comp f _{IJ} 
    \end{align*}
\end{proof}

\subsection{Filtrations}
\label{sub:Filtrations}
In this section we will build on the constructions in the previous section to prove 
that the constructed morphism $\varphi$ is indeed a quasi-isomorphism.
This procedure will heavily rely on the following two Lemmas.
\begin{lmm}\cite[Mapping Lemma 5.2.4]{caw}
    \label{lmm:two}
    Let $f: \lbrace E^r _{pq} \rbrace \rightarrow \lbrace F^r _{pq} \rbrace$ be a 
    morphism of spectral sequences, that is, a family of maps $f^r _{pq} : E^r _{pq} \rightarrow F^r _{pq}$ satisfying 
    \begin{align*}
            d^r f^r = f ^{r+1} d^r \\
            f ^{r+1} = H(f^r)
\end{align*}
    such that moreover for some fixed $r$, $f^r : E^r _{pq} \simeq F^r _{pq} $
    is an isomorphism for all $p$ and $q$.
    The $5$-Lemma implies that $f^s : E^s _{pq} \simeq F^s _{pq} $ for all $s \geq r$ as well.
\end{lmm}
\noindent
A special case of this Lemma is 
\begin{lmm}
    \label{lmm:one}
    Let $f : V \rightarrow W$ be a map of chain complexes $V$ and $W$ and equip both with 
    a bounded above and complete filtration $\mathcal{F} ^p$ compatible with $f$, that is,
    $f(\mathcal{F} ^pV) \subset \mathcal{F} ^pW$.
    Then $f : V \rightarrow W$ is a quasi-isomorphism if $\mathsf{gr} f : \mathsf{gr} V \rightarrow \mathsf{gr} W$ is one.
\end{lmm}
\noindent
The latter lemma can be applied to $\varphi : \mathcal{D} \rightarrow \mathcal{S}$ with a suitable filtration on both sides
because proposition (\ref{prop:dcptl}) implies that it is a chain map. 
By mimicking the construction from \cite{cw} we will construct different filtrations to
sort out certain differentials on both complexes and thereby simplify the verification of $\varphi$ being a quasi-isomorphism.\\
A generic element $T$ in the complex $\mathcal{D}$ consists of a graph $\Gamma \in \mathsf{
Graphs}_M$, elements $c^*$ and $l^*$ of $\mathsf{Com} ^*$ and $\mathsf{Lie} ^*$ 
respectively, as well as polynomials $f_1, \dotsc, f_r \in \mathcal{O}$
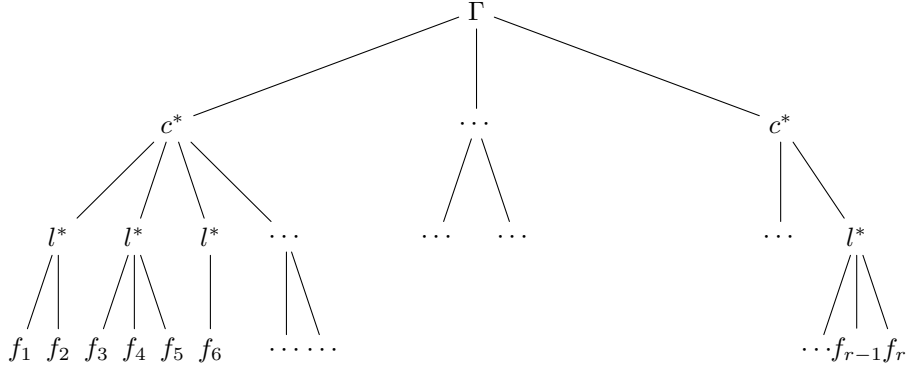
\begin{center}
\begin{tikzpicture}[line cap=round,line join=round,>=triangle 45,x=.5cm,y=.5cm]
\clip(0,0) rectangle (32,11);
\node (g) at (16,10) {$\Gamma$};
\node (c1) at (8,7) {$c^*$};
\node (l11) at (5,4) {$l^*$};
\node (f111) at (4,1) {$f_1$};
\node (f112) at (5,1) {$f_2$};
\node (l12) at (7,4) {$l^*$};
\node (f121) at (6,1) {$f_3$};
\node (f122) at (7,1) {$f_4$};
\node (f123) at (8,1) {$f_5$};
\node (l13) at (9,4) {$l^*$};
\node (f131) at (9,1) {$f_6$};
\node (l14) at (11,4) {$\dotsm$};
\node (f141) at (11,1) {$\dotsm$};
\node (f142) at (12,1) {$\dotsm$};
\node (c2) at (16,7) {$\dotsm$};
\node (l21) at (15,4) {$\dotsm$};
\node (l22) at (17,4) {$\dotsm$};
\node (c3) at (24,7) {$c^*$};
\node (l31) at (24,4) {$\dotsm$};
\node (l32) at (26,4) {$l^*$};
\node (f321) at (25,1) {$\dotsm$};
\node (f322) at (26,1) {$f_{r-1}$};
\node (f323) at (27,1) {$f_r$};
\draw[-, black, font=\normalsize]
(g) edge node[below] {$$} (c1)
(c1) edge node[below] {$$} (l11)
(l11) edge node[below] {$$} (f111)
(l11) edge node[below] {$$} (f112)
(c1) edge node[below] {$$} (l12)
(l12) edge node[below] {$$} (f121)
(l12) edge node[below] {$$} (f122)
(l12) edge node[below] {$$} (f123)
(l14) edge node[below] {$$} (f141)
(l14) edge node[below] {$$} (f142)
(c1) edge node[below] {$$} (l13)
(l13) edge node[below] {$$} (f131)
(c1) edge node[below] {$$} (l14)
(g) edge node[below] {$$} (c2)
(c2) edge node[below] {$$} (l21)
(c2) edge node[below] {$$} (l22)
(g) edge node[below] {$$} (c3)
(c3) edge node[below] {$$} (l31)
(c3) edge node[below] {$$} (l32)
(l32) edge node[below] {$$} (f321)
(l32) edge node[below] {$$} (f322)
(l32) edge node[below] {$$} (f323);
\end{tikzpicture}
\end{center}
Let $l$ be the number of $\mathcal{Lie}^*$ generators in this element and $\mathsf{deg} _{\mathcal{O}} = \sum _{i = 1 } ^{r} \mathsf{deg} _{\mathcal{O}} (f_i)$
the sum of the degrees of all attached polynomials.
This degree is well defined as $\mathsf{deg} _{\mathcal{O}} (f_i \cdot f_j) = \mathsf{deg} _{\mathcal{O}}(f_i) + \mathsf{deg} _{\mathcal{O}}(f_j)$.
Notice that on $\mathcal{D}$ the degree
\begin{align*}
        \mathsf{deg}^1 = r - l - \mathsf{deg} _{\mathcal{O}} 
\end{align*}
is raised by the differentials $d ^{\mathsf{lie}} _{\mathsf{mod}}$ and $d^{\mathsf{com}} _{\mathsf{alg}}$ by one and left unchanged by all other parts.\\
Hence on $\mathsf{gr} ^1 \mathcal{D}$, the graded complex associated to the filtration 
by this degree, the remaining differential is
\begin{align*}
    \delta _{\text{split}} + \delta _{\text{pair}} + \delta^{\mathsf{z}_M}
    + d ^{\mathsf{com}} _{\mathsf{mod}} + d ^{\mathsf{lie}} _{\mathsf{alg}}
\end{align*}
A polynomial $F \in \mathcal{S}$ can equivalently be represented as a graph $\Gamma$ 
with multiple connected components. Each one consists of a single vertex decorated with a cohomology class
and an element $v_i \in V$ and represents a variable $\alpha \otimes v_i$.
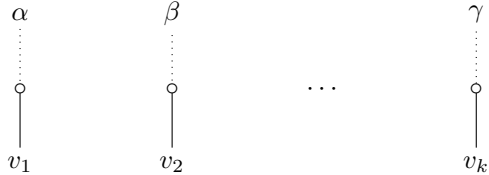
\begin{center}
\begin{tikzpicture}[scale=1]
\node (d1) at (0,3) {$\alpha$};
\node (d2) at (2,3) {$\beta$};
\node (d4) at (6,3) {$\gamma$};
\node (v1) at (0,2) {$\circ$};
\node (v2) at (2,2) {$\circ$};
\node (v3) at (4,2) {$\dotsm$};
\node (v4) at (6,2) {$\circ$};
\node (e1) at (0,1) {$v_1$};
\node (e2) at (2,1) {$v_2$};
\node (e4) at (6,1) {$v_k$};
\draw[-, black, font=\normalsize]
(0,1.95) -- (e1)
(2,1.95) -- (e2)
(6,1.95) -- (e4);
\draw[-,dotted, black, font=\normalsize]
(d1) -- (0,2.05)
(d2) -- (2,2.05)
(d4) -- (6,2.05);
\end{tikzpicture}
\end{center}
\noindent
Hence for the image of an element $T \in \mathcal{D}$ $\mathsf{deg} ^1$ vanishes and
we therefore construct the trivial filtration by the zero-degree, in which 
$\mathcal{D}$ is concentrated in degree zero.
Thereby we obtain that $\varphi$ is compatible with the filtrations and can reduce our
problem to $\mathsf{gr}^1 \varphi : \mathsf{gr}^1 \mathcal{D} \rightarrow \mathsf{gr}^1 \mathcal{S}$.
To move on, let us define another degree on $\mathcal{D}$ , which is the number of edges minus the 
number of vertices of the graph $\Gamma$.
\begin{align*}
        \mathsf{deg}^2 = e - v
\end{align*}
All differentials either increase of leave this degree constant, so we can filter the 
complex by $\mathsf{deg}^2$. On the associated graded complex $\mathsf{gr}^2 \mathsf{gr}^1 \mathcal{D}$
the $\delta _{\text{pair}}$ term in the induced differential vanishes completely.
For $\delta^{\mathsf{z}_M}$ the terms that contribute are the terms attaching components with $e$ edges and $s = e+1$ internal vertices.
On $\mathcal{S}$ we choose the corresponding degree to be minus the polynomial degree
\begin{align*}
        \widetilde{\mathsf{deg}}^2 = -\mathsf{deg}_{\mathcal{S}} 
\end{align*}
such that $\varphi$ is again compatible with these filtrations.
Notice that this eliminates the differential $\Delta$ on $\mathcal{S}$ and we are left 
with the zero differential on $ \mathsf{gr} ^2 \mathsf{gr} ^1 \mathcal{S}$.
All remaining parts of the differential on $\mathcal{D}$ exclusively either increase the number of internal 
vertices $s$ by at least one, decrease the number of $\text{Com} ^*$ elements, which 
we denote by $c$, by at least one or don't affect any of these degrees at all.
So it is well defined to filter by the degree
\begin{align*}
    \mathsf{deg} ^3 = s + \left( n - 1 \right) \cdot c - \mathsf{deg}
\end{align*}
where $\mathsf{deg}$ is the degree of $\Gamma$ defined in equation (\ref{eq:graphdeg}).
On the associated graded complex $\mathsf{gr} ^3 \mathsf{gr} ^2 \mathsf{gr} ^1 \mathcal{D}$
only the differentials $\delta _{\text{split}}$, $d^{\mathsf{com}}_{\mathsf{mod}} $ and $d^{\mathsf{lie}}_{\mathsf{alg}}$ remain unchanged. By 
construction only the terms of $\delta^{\mathsf{z}_M}$, which add exactly one edge and one internal vertex remain.
That is either the term replacing a decoration by an internal vertex with the same decoration or connecting an internal vertex to an edge.
In the latter case the vertex can be considered to be decorated with the unit element $\mathbb{I} \in \overline{H(M)}$.
Both terms appear as well in $\delta _{\text{split}} $ with the opposite sign and therefore cancel out.
Because an element of $\mathcal{S}$, as illustrated above, has neither 
edges nor internal vertices or $\mathsf{Com}^*$ elements, we can once again use the 
trivial filtration on $\mathcal{S}$ to achieve compatibility of $\varphi$.
\\
Lastly we use the filtration by the degree $\mathsf{deg} ^4 = -c$ to construct on
$\mathcal{D}$ the naturally determined spectral sequence $E _{pq}$, starting with $E^0 _{pq} = \QR{\mathcal{F} ^p \mathcal{D}  _{p+q} }{\mathcal{F} ^{p-1} \mathcal{D}  _{p+q} }$.
This splits the differential in the following way
\begin{align*}
        d = \underbrace{\delta_{\text{split}} + d^{\mathsf{lie}}_{\mathsf{alg}} }_{=d^0} + \underbrace{d^{\mathsf{com}}_{\mathsf{mod}}}_{=d^1}
\end{align*}
One easily checks that for the trivial spectral sequence $F _{pq}$ in which
$\mathcal{S}$ is concentrated in one column $p=0$ we obtain that the collection of 
maps $\varphi^0 _{pq} : E^0 _{pq} \rightarrow F^0 _{pq}$
is a morphism between the $0^{th}$ page of these spectral sequences.
To compute the first page $E^1 _{pq}$ of this spectral sequence, notice 
that the terms of the differential on the zeroth page act in the following way.
\begin{center}
\begin{tikzpicture}[scale=0.9]
    \node (k0) at (0,0) {$\mathsf{Graphs} _M$};
    \node (k1) at (1,0) {$\comp $};
    \node (k2) at (2,0) {$\mathsf{Com}^* \left[ n \right] $};
    \node (k3) at (3,0) {$\comp $};
    \node (k4) at (3.9,0) {$\mathsf{Lie}^* \left[ 1 \right] $};
    \node (k5) at (4.8,0) {$\comp $};
    \node (k6) at (5.5,0) {$\mathsf{Com}$};
    \node (k7) at (6.2,0) {$\comp $};
    \node (k8) at (6.6,0) {$V$};
    \draw[->, black, font=\normalsize]
    (k0) edge[loop above, out=120, in=60, looseness=8] node[above] {$\delta_{\text{split}}$} (k0)
    (k4) edge[bend left, out=80, in=100, looseness=1.4] node[above] {$d^{\mathsf{lie}}_{\mathsf{alg}} $} (k6);
\end{tikzpicture}
\end{center}
In this case we can use the commutativity of the operadic composition and the homology functor $H _{\ast}$.
\begin{align}
\begin{split}
    \label{eq:fp}
        &H _{\ast}  \left( \mathsf{Graphs} _M \comp \mathsf{Com}^* \left[ n \right]  \comp \mathsf{Lie}^* \left[ 1 \right]  \comp \mathsf{Com} \comp V, \delta _{\text{split}} + d ^{\mathsf{lie}} _{\mathsf{alg}} \right) \\
= &H _{\ast}  \left( \mathsf{Graphs} _M , \delta _{\text{split}} \right) \comp \mathsf{Com}^* \left[ n \right] \comp \underbrace{H \left( \mathsf{Lie}^* \left[ 1 \right]  \comp \mathsf{Com} \comp V, d ^{\mathsf{lie}} _{\mathsf{alg}} \right)}_{=V} \\
\end{split}
\end{align}
The latter term of the homology can be directly simplified, as this is the Harrison complex 
of $\mathcal{O} = \mathsf{Com} \comp V$, which is acyclic.
For the other term we need to introduce more tools and construct a nested spectral sequence.

\subsection{The Homology of $\mathsf{Graphs}_M$}
\label{sub:The Homology of GraphsM}
In order to compute the homology of $\left( \mathsf{Graphs}_M, \delta _{\text{split}}\right)$,
we introduce the following cdga model which was similarly presented by Lambrechts and Stanley in \cite{ls}.
While the relations (1)-(3) where already discovered by Cohen, this model was first 
found independently by Kriz \cite{kr}  and Totaro \cite{to} in different ways.
A Poincar\'e duality algebra $A$ of formal dimension $n$ is a graded commutative $\bf{k}$-algebra over a field $\bf{k}$ with a non degenerated pairing $A ^{k} \otimes A ^{n-k} \rightarrow \bf{k}$.
Choose a basis $\lbrace \alpha ^j \rbrace _{j = 1} ^N$ with its Poincar\'e dual basis $\lbrace \alpha ^{*j}  \rbrace _{j = 1} ^N$.
\begin{ddef}
    \label{ddef:ls}
    \hfill
    \begin{enumerate}
    \item 
    For a Poincar\'e duality algebra $A$ of formal dimension $n$, the  k-configuration 
            model $\left( \mathcal{F}^* _{A}(k), d^* _{\mathcal{F}} \right) $ is the free graded commutative differential graded algebra generated by
    \begin{equation*}
    \begin{split}
        \label{eq:}
            &\omega_{ab} \qquad 1 \leq a \neq b \leq r  \qquad \text{ of degree } (n-1) \\
            &\alpha _a ^i \qquad \lbrace \alpha^i \rbrace \text{ basis of } A
    \end{split}
    \end{equation*}
    subject to the relations
    \begin{align*}
        &\text{(i) skew symmetry}            & &\omega _{ab}  = (-1) ^{n} \omega _{ba} \\
        &\text{(ii) Arnold identity }        & &\omega _{ab} \omega _{bc} + \omega _{bc} \omega _{ca} + \omega _{ca} \omega _{ab} = 0 \\
        &\text{(iii) move decorations}       & &\omega _{ab} \alpha _a^p = \omega _{ab} \alpha_b^p \\
        &\text{(iv) multiply decorations }   & &\alpha^i_a \alpha^j_a = \left( \alpha^i \cdot \alpha^j \right) _a
    \end{align*}
    with the differential $d _{\mathcal{F}}^*~\omega _{ab} = \sum _{j = 1} ^N \alpha_a ^{j} \alpha_b ^{j*} $.
            \vspace{4pt}
    \item 
        The cdga's $\mathcal{F}^* _{A}(k)$ assmble together to the $\mathsf{e}_n ^{\vee} $-module $\mathcal{F}^* _{A}$.
    \end{enumerate}
\end{ddef}
\noindent
Using the paper \cite{ds} by Sinha, we will now shed some light on this construction.
Disregarding for a moment the generators $\alpha ^j$ and the relations $(iii)$ and $(iv)$, we obtain the complex
\begin{align*}
    \mathsf{Siop} ^n = \QR{~^* \mathsf{Gra} ^n}{\Big< \text{Skew} + \text{Arnold} \Big>}
\end{align*}
where an element $\omega _{ab}$ represents an edge between the vertices $a$ and $b$.
These relations are graphically represented by
\begin{align*}
    \raisebox{-4.7ex}{
    \begin{tikzpicture}[scale=0.4]
        \clip(0,0) rectangle (4,4);
        \node (v1) at (.5,2) {$1$}; 
        \node (v2) at (2.5,2) {$2$}; 
        \draw
        (v1) circle (11pt)
        (v2) circle (11pt);
        \draw[->] (.9,2) -- (2.1,2);
    \end{tikzpicture}}
    + (-1)^n~~
    \raisebox{-4.7ex}{
    \begin{tikzpicture}[scale=0.4]
        \clip(0,0) rectangle (4,4);
        \node (v1) at (.5,2) {$1$}; 
        \node (v2) at (2.5,2) {$2$}; 
        \draw
        (v1) circle (11pt)
        (v2) circle (11pt);
        \draw[->] (2.1,2) -- (.9,2);
    \end{tikzpicture}} = 0
    \qquad \qquad \qquad 
    \raisebox{-3.0ex}{
    \begin{tikzpicture}[scale=0.4]
        \clip(0,0) rectangle (4,4);
        \node (v1) at (.5,0.5) {$1$}; 
        \node (v2) at (1.5,2.23) {$2$}; 
        \node (v3) at (2.5,.5) {$3$}; 
        \draw
        (v1) circle (11pt)
        (v2) circle (11pt)
        (v3) circle (11pt);
        \draw[->] (.7,.85) -- (1.3,1.89); % 1-2
        \draw[->] (1.7,1.89) -- (2.31,.84); % 2-3
    \end{tikzpicture}}
    +
    \raisebox{-3.0ex}{
    \begin{tikzpicture}[scale=0.4]
        \clip(0,0) rectangle (4,4);
        \node (v1) at (.5,0.5) {$1$}; 
        \node (v2) at (1.5,2.23) {$2$}; 
        \node (v3) at (2.5,.5) {$3$}; 
        \draw
        (v1) circle (11pt)
        (v2) circle (11pt)
        (v3) circle (11pt);
        \draw[->] (.7,.85) -- (1.3,1.89); % 1-2
        \draw[->] (2.1,.5) -- (.9,.5); % 3-1
    \end{tikzpicture}}
    +
    \raisebox{-3.0ex}{
    \begin{tikzpicture}[scale=0.4]
        \clip(0,0) rectangle (4,4);
        \node (v1) at (.5,0.5) {$1$}; 
        \node (v2) at (1.5,2.23) {$2$}; 
        \node (v3) at (2.5,.5) {$3$}; 
        \draw
        (v1) circle (11pt)
        (v2) circle (11pt)
        (v3) circle (11pt);
        \draw[->] (1.7,1.89) -- (2.31,.84); % 2-3
        \draw[->] (2.1,.5) -- (.9,.5); % 3-1
    \end{tikzpicture}}
    = 0
\end{align*}
Furthermore, we define
\begin{align*}
        \mathsf{Pois} ^n = \QR{\text{n-forests}}{\Big<\text{Skew} + \text{Jacobi} \Big>}
\end{align*}
the quotient of the free module spanned by n-forests, that is, a collection of trees 
with univalent or trivalent vertices whose set of leaves is in bijection with the 
integers $\lbrace 1, \dotsc, n\rbrace $, by skew-symmetry and the Jacobi identity.
Here every branch point stands for the Lie bracket evaluated on its leaves.
$\mathsf{Pois}^n$ has a basis consisting of n-forests with tall trees (the distance between the 
leaf with the minimal label and the root in each tree is maximal) and $\mathsf{Siop}^n$ one 
consisting of long graphs (each component is a linear graph starting with the minimal label).
Consider the map $\beta _{\Gamma, F} : E _{\Gamma} \rightarrow I _{F}$ from the set of edges 
of an n-graph $\Gamma$ to the set of internal bifurcations of the trees in an n-forest $F$ that sends 
an edge directed from $i$ to $j$ to the lowest branch point on the shortest path from 
the $i^{th} $ to the $j^{th} $ leaf (if these leaves are attached to the same tree).
With this map one can introduce a pairing between n-graphs and n-forests
\begin{align*}
        \langle \Gamma, F \rangle = \begin{cases}
            \pm 1 \qquad &\text{ if } \beta _{\Gamma, F} \text{ is a bijection } \\
            0 &\text{otherwise} 
        \end{cases}
\end{align*}
where the sign depends on the ordering of the leaves on the tree. This pairing passes 
on to a perfect pairing between the quotients Pois $^n$ and Siop $^{n}$. Hence the induced map
$ \tilde \beta : \mathsf{Pois} ^n \xrightarrow{\sim} \left( \mathsf{Siop} ^n \right) ^*$
is an isomorphism and we conclude from $\mathsf{Pois}^n = \mathsf{Com} \comp \mathsf{Lie} \left[ n-1 \right] $ 
that $\mathsf{Siop}^n = \mathsf{Com} ^{*}\left[ n \right]  \comp \mathsf{Lie} ^{*} \left[ 1 \right] $.
An element $\alpha_a$ decorates the vertex $a$ with the cohomology class 
$\alpha$. Due to relation $(iii)$ one can move a decoration between vertices 
connected by an edge and by relation $(iv)$ multiply two such decorations attached to the same vertex.
Taking the Poincar\'e duality algebra to be $A = H_*(M)$, this implies that an element 
of $\mathcal{F}^*_{H_*(M)}$ represents a graph 
without internal vertices, decorated with cohomology classes at each connected 
component.
However, we will work with the dual model $\mathcal{F} _{H^*(M)} = \left( \mathcal{F}^* _{H_*(M)} \right)^*$ which is the
complex of decorated n-forests, so we obtain the $\mathsf{e}_n$-module
\begin{align*}
        \mathcal{F} _{H^*(M)} = \mathsf{Com} \comp H^* (M) \comp \mathsf{Lie} \left[ n-1 \right] 
\end{align*}
with the usual composition of trees and the dual differential $d_{\mathcal{F}} = (d_{\mathcal{F}}^*)^*$.
\begin{lmm}
    \label{lmm:rec}
    The model $\mathcal{F} _{H^*(M)}$ satisfies the recursive identity
    \begin{align*}
           \mathcal{F} _{H^*(M)}(k)  = \mathcal{F} _{H^*(M)}(k-1)  \otimes H _{\ast} (M) \oplus \mathcal{F} _{H^*(M)}(k-1) \left[n-1\right]^{\oplus(k-1)} 
    \end{align*}
\end{lmm}
\begin{proof}
    We will show this property in the dual setting, which will imply the claim about $\mathcal{F} _{H^*(M)} $.
    One can split $U(r) := \mathcal{F}^* _{H_*(M)} (r) $ into a direct sum
    \begin{align}
    \begin{split}
        \label{eq:split}
            U(r) = U_0(r) \oplus U_1(r) \oplus U_2(r)
    \end{split}
\end{align}
where $U_i$ contains all graphs in which the vertex $1$ has degree $i$, counting 
decorations. For the case of an isolated vertex, it can be easily seen that $U_0(r) \simeq U(r-1)$.
If we allow vertex $1$ to have one edge, this can either be a decoration, then we obtain a term isomorphic to
$U(r-1) \otimes \overline{H}_* (M)$, or an edge to the remaining part of the graph, in which case the term is
isomorphic to $U(r-1)\left[ n-1 \right] ^{\oplus(r-1)}$. In the basis of long graphs the first vertex can have at most valence 2
of which one is due to a decoration. However this case is equivalent to $U_1$, since one can move the decoration
with relation $(iii)$ in definition (\ref{ddef:ls}) along a connected component.
In total this amounts to
\begin{align*}
        U(r) &= U(r-1) \oplus U(r-1) \otimes \overline{H}_*(M) \oplus U(r-1)\left[ n-1 \right]  ^{\oplus (r-1)} \\
        &= U(r-1) \otimes H_* (M) \oplus U(r-1)\left[ n-1 \right]  ^{\oplus (r-1)} 
\end{align*}
\end{proof}
\noindent
Now we are able to conclude the desired result of this subsection, which is to express the homology of $\mathsf{Graphs}_M$ up to quasi-isomorphism by the model $\mathcal{F} _{H^*(M)}$.
Since its differential $d_{\mathcal{F}}$ would naturally correspond to $\delta _{\text{pair}}$
which however vanished through the second filtration, we consider $\mathcal{F} _{H^*(M)}$ without any differential.
\begin{prop}
    \label{prop:quasi}
    The map
    \begin{align*}
            \mathcal{F} _{H^*(M)} \xrightarrow{\sim} H_{\ast} \left( \mathsf{Graphs}_M , \delta_{\text{split}} \right)  
    \end{align*}
    is an isomorphism of graded $\mathsf{e}_n$-modules.
\end{prop}
\begin{proof}
    The embedding $\mathcal{F} _{H^*(M)} \hookrightarrow \mathsf{Graphs}_M$ defined on the generators by
    \begin{align}
    \begin{split}
        \label{eq:inc}
        \lbrace ~,~\rbrace~ &\mapsto ~~~\text{$\circ$-$\circ$} \\
        \wedge~~ &\mapsto~~~ \text{$\circ$~$\circ$} 
    \end{split}
    \end{align}
    is a morphism of operadic $\mathsf{e}_n$-modules.
    We will prove the claim by induction on the arity $r$.  The result holds trivially 
    for the case $\mathcal{F} _{H^*(M)}(0) \simeq H_*(\mathsf{Graphs}_M(0), \delta_{\text{split}})$.
    Split the complex $\mathsf{Graphs}_M (r)$ into $C(r) = C_0 \oplus C_1 \oplus C_{\geq 2}$
    in the same way as $\mathcal{F}^* _{H_*(M)}$ was split in equation (\ref{eq:split}).
    For $C_0$ we observe $C_0 \simeq C(n-1)$.
    The differential $\delta _{\text{split}}$ on $\mathsf{Graphs}_M$ has the following components.
    \begin{center}
    \begin{tikzpicture}[scale=1]
    \label{tikz:diff}
        \node (k0) at (0,0) {$C_0$};
        \node (p1) at (1,0) {$\oplus$};
        \node (k1) at (2,0) {$C_1$};
        \node (p2) at (3,0) {$\oplus$};
        \node (k2) at (4,0) {$C _{\geq 2} $};
        \draw[->, black, font=\normalsize]
        (k0) edge[loop above, out=120, in=60, looseness=8] node[above] {$$} (k0)
        (k1) edge[loop above, out=120, in=60, looseness=8] node[above] {$$} (k1)
        (k2) edge[loop above, out=120, in=60, looseness=8] node[above] {$$} (k2)
        %(k2) edge[bend right] node[above] {$d^{com}_{mod} $} (k0)
        (k2) edge[bend left, out=60, in=120, looseness=1] node[below] {$\delta_{\text{split}} ^0$} (k1);
    \end{tikzpicture}
    \end{center}
    With the descending filtration $\mathcal{F}^pC(r) = \lbrace \Gamma \in C _{\geq 2}
    ~\vert~ \mathsf{deg} (\Gamma) \geq k \rbrace \oplus \lbrace \Gamma \in C _1 ~\vert~ \mathsf{deg} (\Gamma) \geq k+1 \rbrace$
    on $C(r)$ we set up a spectral sequence in which the part of the differential on the 0$^{th}$ page
    is $\delta _{\text{split}} ^0$.
    Since this differential is injective on $C_{\geq 2}$, the kernel is exactly $\text{ker}(\delta _{\text{split}} ^0) = C_1$.
    The next page of the spectral sequence is therfore
    \begin{align*}
        E^1 (C(r)) = H(\mathsf{gr} C(r)) = C_0 \oplus \frac{\text{ker}(\delta _{\text{split}} ^0)}{\text{im}(\delta _{\text{split}} ^0)}
        = \frac{C_1}{\text{im}(\delta _{\text{split}} ^0)} = \text{coker}(\delta _{\text{split}}^0)
    \end{align*}
    Only two kinds of graphs can appear in this quotient, namely a graph with
    \begin{enumerate}
        \item[(a)] an isolated vertex 1 that is decorated
        \item[(b)] a vertex 1 connected to another vertex that is not decorated
    \end{enumerate}
    By the same argument as in lemma (\ref{lmm:rec}) we obtain for both cases one term in the following sum
    \begin{align*}
            E^1 (C(r)) 
            &= C_0 \oplus C(r-1) \otimes \Bar{H}^* (M) \oplus C(r-1) [n-1] ^{\oplus (r-1)} \\
            &= C(r-1) \otimes H^* (M) \oplus C(r-1) [n-1] ^{\oplus (r-1)} \\
    \end{align*}
    The differential on this page concides with the differential on $C(r-1)$.
    This spectral sequence is a double complex concentrated on a double coloumn and therefore it collapses on the second page.
    Hence we obtain the recursive identity 
    \begin{align*}
            H(C(r)) = H(C(r-1)) \otimes H^* (M) \oplus H(C(r-1)) [n-1] ^{\oplus (r-1)}
    \end{align*}
    which is the same as for $\mathcal{F} _{H_*(M)}$.
    By the induction hypothesis this concludes the proof.
\end{proof}
\noindent
With this result we continue the computation from equation (\ref{eq:fp}) and obtain on the first page.
\begin{align}
    \label{eq:fp2}
    \begin{split}
        E^1 
        &= H_{\ast} \left( \mathsf{Graphs}_M , \delta_{\text{split}} \right) \comp \mathsf{Com}^* \left[ n \right] \comp V\\
        &\simeq \mathcal{F} _{H^*(M)} \comp \mathsf{Com}^* \left[ n \right] \comp V\\
        &= \mathsf{Com} \comp H^*(M) \comp \mathsf{Lie} \left[ n-1 \right] \comp \mathsf{Com}^* \left[ n \right] \comp V\\
    \end{split}
\end{align}
The differential on this page is $d^1 = d^{\mathsf{com}}_{\mathsf{mod}}$.
To proceed we verify that $d^{\mathsf{com}}_{\mathsf{mod}}$ corresponds to a canonically induced twisting differential $d _{\kappa}$.
\begin{lmm}
\label{lmm:isomod}
The isomorphism from lemma (\ref{prop:quasi}) induces an isomorphism of differential graded $\mathsf{e}_n$-modules
\begin{align*}
    \left(\mathcal{F} _{H_*(M)}  \comp \mathsf{Com} ^* \left[ n \right]  \comp V, d _{\kappa}\right)
    \xrightarrow{\sim} \left( H _{\ast} (\mathsf{Graphs} _M, \delta _{split} ) \comp \mathsf{Com} ^* \left[ n \right]  \comp V, d^{\mathsf{com}}_{\mathsf{mod}} \right)
\end{align*}
\end{lmm}
\begin{proof}
    One needs to verify that the map is compatible with the differentials.
    We will omit the terms $\mathsf{Com} \comp H(M)$ in $\mathcal{F} _{H_*(M)}$ as 
    only the terms $\mathsf{Lie} \comp \mathsf{Com}^*$ play an important role.
    Since $\mathsf{Com}^* = \mathsf{Lie} ^{\vee} $ is the Koszul dual operad of 
    $\mathsf{Lie} $ there exists a natural twisting morphism
    \begin{align*}
            \kappa : \mathsf{Com}^* \left[ n \right]  \rightarrow \mathsf{Lie} \left[ n-1 \right] 
    \end{align*}
    of degree $-1$, sending the generator of $\mathsf{Com} ^*$ to the Lie bracket, the 
    generator of $\mathsf{Lie} $. With this we construct the unique differential $d _{\kappa} $ on 
    the Koszul complex $\left(\mathsf{Lie} \comp  \mathsf{Com} ^*, d _{\kappa} \right)$ as
    \begin{center}
    \begin{tikzpicture}[scale=1.0]
    \node (k0) at (-0.4,1) {$d_{\kappa} ~ : $};
    \node (k1) at (1,1) {$\mathsf{Lie} \comp \mathsf{Com} ^* $};
    \node (k2) at (6,1) {$\mathsf{Lie} \comp \left( \mathsf{Com}^* ; \mathsf{Com} ^* \comp \mathsf{Com} ^* \right)  $};
    \node (k3) at (3,0) {$\simeq \left(\mathsf{Lie} \comp _{(1)} \mathsf{Com} ^* \right) \comp \mathsf{Com} ^*$};
    \node (k4) at (10,0) {$\left( \mathsf{Lie} \comp _{(1)} \mathsf{Lie}  \right) \comp \mathsf{Com} ^*$};
    \node (k5) at (15,0) {$\mathsf{Lie} \comp \mathsf{Com}^*$};
    \draw[->, black, font=\normalsize]
    (k1) edge node[above] {$\text{id}_{\mathsf{Lie}} \comp ' \Delta$} (k2)
    (k3) edge node[above] {$\left( \id _{\mathsf{Lie}} \comp _{(1)} \kappa \right) \circ \id _{\mathsf{Com}^*} $} (k4)
    (k4) edge node[above] {$\mu _{(1)} \comp \id _{\mathsf{Com} ^*}$} (k5);
    \end{tikzpicture}
    \end{center}
where $\Delta$ denotes only in this proof the decomposition map of the cooperad $\mathsf{Com}^*$.
The symbols $\comp _{(1)}$, $\mu _{(1)}$ and $\comp '$ refer to the infinitesimal composite, the 
infinitesimal composition map and the infinitesimal composite of morphisms respectively, which are explained in chapter 6 of Loday's and Valette's book \cite{lv}.
On $\mathsf{Graphs} _M \comp \mathsf{Com}^* \left[ n \right] $ the differential $d^{\mathsf{com}} _{\mathsf{mod}}$ is 
constructed in a very similar fashion using instead of the twisting morphism $\kappa$ 
the composition $\iota \kappa$ with the inclusion from equation (\ref{eq:inc}).
This time we represent this mechanism graphically.
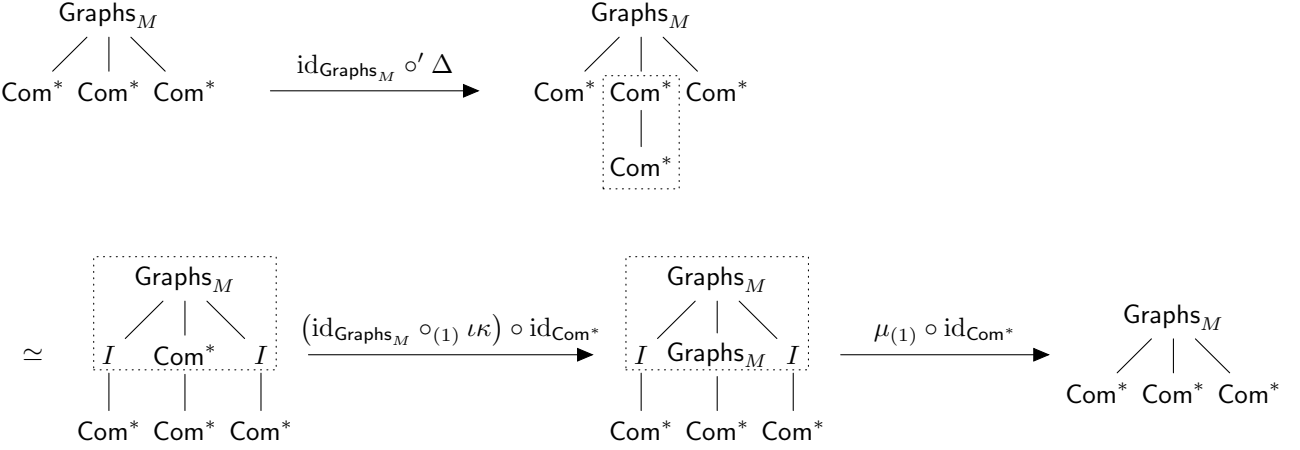
\begin{tikzpicture}[line cap=round,line join=round,>=triangle 45,x=1cm,y=1cm]
\node (g1) at (2,5.5) {$\mathsf{Graphs}_M$};
\node (c11) at (1,4.5) {$\mathsf{Com}^*$};
\node (c12) at (2,4.5) {$\mathsf{Com}^*$};
\node (c13) at (3,4.5) {$\mathsf{Com}^*$};
\node (r1) at (4,4.5) {$$};
\node (l2) at (7,4.5) {$$};
\node (g2) at (9,5.5) {$\mathsf{Graphs}_M$};
\node (c21) at (8,4.5) {$\mathsf{Com}^*$};
\node (c22) at (9,4.5) {$\mathsf{Com}^*$};
\node (c221) at (9,3.5) {$\mathsf{Com}^*$};
\node (c23) at (10,4.5) {$\mathsf{Com}^*$};
\draw[black,dotted] (8.5,3.2) -- (8.5,4.7) -- (9.5,4.7) -- (9.5,3.2) -- cycle;
\node (seq) at (1,1) {$\simeq$};
\node (g3) at (3,2) {$\mathsf{Graphs}_M$};
\node (c31) at (2,1) {$I$};
\node (c32) at (3,1) {$\mathsf{Com}^*$};
\node (c33) at (4,1) {$I$};
\node (c311) at (2,0) {$\mathsf{Com}^*$};
\node (c321) at (3,0) {$\mathsf{Com}^*$};
\node (c331) at (4,0) {$\mathsf{Com}^*$};
\draw[black, dotted] (1.8,.8) -- (1.8,2.3) -- (4.2,2.3) -- (4.2,.8) -- cycle;
\node (r3) at (4.5,1) {$$};
\node (l4) at (8.5,1) {$$};
\node (g4) at (10,2) {$\mathsf{Graphs}_M$};
\node (c41) at (9,1) {$I$};
\node (c42) at (10,1) {$\mathsf{Graphs}_M$};
\node (c43) at (11,1) {$I$};
\node (c411) at (9,0) {$\mathsf{Com}^*$};
\node (c421) at (10,0) {$\mathsf{Com}^*$};
\node (c431) at (11,0) {$\mathsf{Com}^*$};
\draw[black, dotted] (8.8,.8) -- (8.8,2.3) -- (11.2,2.3) -- (11.2,.8) -- cycle;
\node (r4) at (11.5,1) {$$};
\node (l5) at (14.5,1) {$$};
\node (g5) at (16,1.5) {$\mathsf{Graphs}_M$};
\node (c51) at (15,0.5) {$\mathsf{Com}^*$};
\node (c52) at (16,0.5) {$\mathsf{Com}^*$};
\node (c53) at (17,0.5) {$\mathsf{Com}^*$};
\draw[-, black, font=\normalsize]
(g1) edge node[below] {$$} (c11)
(g1) edge node[below] {$$} (c12)
(g1) edge node[below] {$$} (c13)
(g2) edge node[below] {$$} (c21)
(g2) edge node[below] {$$} (c22)
(g2) edge node[below] {$$} (c23)
(c22) edge node[below] {$$} (c221)
(g3) edge node[below] {$$} (c31)
(g3) edge node[below] {$$} (c32)
(g3) edge node[below] {$$} (c33)
(c31) edge node[below] {$$} (c311)
(c32) edge node[below] {$$} (c321)
(c33) edge node[below] {$$} (c331)
(g4) edge node[below] {$$} (c41)
(g4) edge node[below] {$$} (c42)
(g4) edge node[below] {$$} (c43)
(c41) edge node[below] {$$} (c411)
(c42) edge node[below] {$$} (c421)
(c43) edge node[below] {$$} (c431)
(g5) edge node[below] {$$} (c51)
(g5) edge node[below] {$$} (c52)
(g5) edge node[below] {$$} (c53);
\draw[->, black, font=\normalsize]
(r1) edge node[above] {$\id _{\mathsf{Graphs}_M} \comp '\Delta$} (l2)
(r3) edge node[above] {$\left( \id _{\mathsf{Graphs}_M} \comp _{(1)} \iota \kappa \right) \comp \id _{\mathsf{Com}^*} $} (l4)
(r4) edge node[above] {$\mu _{(1)} \comp \id _{\mathsf{Com}^*}$} (l5);
%\draw[->] (c21) .. controls ([xshift=-4cm] c21) and ([xshift=4cm] g3) .. (g3);
\end{tikzpicture}

\noindent
Thus using the fact that $\iota$ and
$\text{id}_{\mathsf{Com}^*}$ are morphisms of operads and cooperads respectively, we can 
verify the claim directly
\begin{align*}
\Big( \iota \comp \text{id} _{\mathsf{Com}^*} \Big) d _{\kappa}
&= \Big( \iota \comp \text{id} _{\mathsf{Com}^*} \Big) \left( \mu^{\mathsf{Lie}} _{(1)} \comp \text{id} _{\mathsf{Com}^*} \right) \Big( \text{id} _{\mathsf{Lie}} \comp ' \big[ \left( \kappa \comp \text{id} _{\mathsf{Com}^*} \right) \Delta \big]  \Big) \\
&= \left( \mu^{\mathsf{Graphs}_M }  _{(1)} \comp \text{id} _{\mathsf{Com}^*} \right) \Big( \left( \iota \comp_{(1)} \iota  \right) \comp \text{id} _{\mathsf{Com}^*} \Big)
    \Big( \left( \text{id} _{\mathsf{Lie}} \comp _{(1)} \kappa \right) \comp \text{id} _{\mathsf{Com}^*} \Big) \Big(\text{id}_{\mathsf{Lie}} \comp' \Delta \Big) \\
&= \left( \mu^{\mathsf{Graphs}_M }  _{(1)} \comp \text{id} _{\mathsf{Com}^*} \right) \Big( \left( \text{id}_{\mathsf{Graphs}_M}  \comp_{(1)} \iota \kappa \right) \comp \text{id} _{\mathsf{Com}^*} \Big)
    \Big( \left( \iota \comp_{(1)} \text{id} _{\mathsf{Com}^*} \right) \comp \text{id} _{\mathsf{Com}^*} \Big) \Big(\text{id}_{\mathsf{Lie}} \comp' \Delta \Big) \\
&= \left( \mu^{\mathsf{Graphs}_M }  _{(1)} \comp \text{id} _{\mathsf{Com}^*} \right) \Big( \left( \text{id}_{\mathsf{Graphs}_M}  \comp _{(1)}  \iota \kappa \right) \comp \text{id} _{\mathsf{Com}^*} \Big)
    \Big(\text{id}_{\mathsf{Graphs}_M} \comp' \Delta \Big) \Big( \iota \comp \text{id} _{\mathsf{Com}^*} \Big) \\
&= d^{\mathsf{com}}_{\mathsf{mod}} \Big( \iota \comp \text{id} _{\mathsf{Com}^*} \Big)
\end{align*}
\end{proof}
\noindent
We therefore denote (by abuse of notation) from now on both differentials by $d^{\mathsf{com}}_{\mathsf{mod}}$.
\subsection{Conclusion}
\label{sub:Conclusion}

With the results from the previous section, we are now able to infer the proof.
\begin{proof}[Proof of Theorem \ref{thm:main1}]
Applying once again Koszul duality results the next page of the spectral sequence from equation (\ref{eq:fp2}) .
\begin{align*}
        E^2 &= H_{\ast} (\mathsf{Com} \comp H^{\ast}(M) \comp \mathsf{Lie}\left[ n-1 \right] 
        \comp \mathsf{Com}^* \left[ n \right] \comp V, d^{\mathsf{com}}_{\mathsf{mod}}) \\
        &\simeq \mathsf{Com} \comp H^{\ast}(M) \comp V
\end{align*}
At this point we can conclude that the morphism $\varphi^2 : E^2 \rightarrow F^2$ on the second page is an 
isomorphism and therefore every induced $\varphi ^r$ for $r \geq 2$ 
will be an isomorphism by lemma (\ref{lmm:two}) as well. 
Consequently, as these spectral sequences converge, given by the fact that they are bounded, $H _{\ast} (\varphi)$ is an isomorphism and
$\varphi$ is a quasi-isomoprhism as claimed.
\hfill
\end{proof}

\section{The Twisted Polynomial Algebra}
\label{sec:twisted}
\noindent
As already explained in section (\ref{sub:Twisting}), one can twist the polynomial 
algebra $\mathcal{O}$ by a Maurer-Cartan element $m$ to obtain $\mathcal{O} \llbracket \hbar \rrbracket$.
In the following we will show how the above computation can be altered to comprise 
this broader case where the additional differential $d^m$ on $\mathcal{O} \llbracket \hbar \rrbracket$ appears.
\begin{thm}
\label{thm:main2}
    Let $M$ be a compact and oriented manifold $M$ with trivialized tangend bundle.
    Let moreover $V$ be the shifted cotangent bundle $V = T^* [1-n] \mathbb{R}^N$ and consider the 
    the polynomial algebra $\mathcal{O} = \mathcal{O} \left( V \right)$.
    Then the factorization homology 
    $\int_M \mathcal{O} \llbracket \hbar \rrbracket$ of $M$
    with coefficients in the twisted polynomial algebra
    $\mathcal{O} \llbracket \hbar \rrbracket = \mathcal{O} ^{\hbar m_1 + \hbar^2 m_2 + \dotsm}$
    is weakly equivalent to the 
    algebra of twisted polynomials $S(H(M) \otimes V)$.
\end{thm}
\noindent
First we add the twisting differential $d^m$ to the composition
\begin{align*}
    \left( \mathcal{D}^m, \partial^m \right)
    := \left( \mathsf{Graphs}_M~\widehat{\comp}_{\mathsf{e}_n}~\mathcal{O} \llbracket \hbar \rrbracket, \partial + d^m \right)
    = \left(\mathsf{Graphs}_M, \delta \right)~\widehat{\comp} _{\mathsf{e}_n}~\left(\mathcal{O}\llbracket \hbar \rrbracket, d^m \right)
\end{align*}
The revised map $\varphi^m : \left( \mathcal{D}^m, \partial ^m \right) \rightarrow \left( \mathcal{S}, \Delta \right)$
now recognizes internal vertices and assigns the Maurer-Cartan element $m$ to all of 
them when $\varphi^m$ is applied. For convenience we use again the notation
$\varphi^m \left( \Gamma \otimes f_I \right)  = \Gamma \comp _m f_I$.
Denote the image of the partition function as $\varphi^m (\mathsf{z}_M) = \mathsf{Z}_M^m$.
In contrast to proposition (\ref{prop:dcptl}), $\delta _{\text{split}}$ now corresponds to
$d^m$ on $\mathcal{O} \llbracket \hbar \rrbracket$.
\begin{lmm}
\label{lmm:ds}
    It holds $\left( \delta_{\text{split}} \Gamma  \right) \comp_m f_I = (-1) ^{\vert \Gamma \vert + 1} \Gamma \comp_m d^m f_I  $
\end{lmm}
\begin{proof}
    It is sufficient to prove this statement for the first summand of each differential.
    \begin{align*}
            \Gamma \comp_m d^m_{(1)}  f_I
            &= \Gamma \comp_m  \big( d^m f_1 \otimes \dotsc \otimes f_k \big) \\
            &= \Gamma \comp_m \big( - \lbrace m, f_1 \rbrace \otimes \dotsc \otimes f_k \big) \\
            &= - \big( \Gamma \comp _1 \text{$\circ$-$\bullet$} \big) \comp_m \big( f_1 \otimes \dotsc \otimes f_k \big) \\
            &= (-1) ^{\vert \Gamma \vert + 1}  \left( \delta_{\text{split}, (1)} \Gamma \right) \comp_m
            \big( f_1 \otimes \dotsc \otimes f_k \big)
    \end{align*}
\end{proof}
\noindent
With this we can show 
\begin{lmm}
    A Maurer-Cartan element $m \in \mathcal{O}$, leads to a Maurer-Cartan element $\mathsf{Z}_M^m \in \mathcal{S}$.
\end{lmm}
\begin{proof}
    We combine the assumptions that $m$ satisfies the Maurer-Cartan equation $\lbrace m, m \rbrace = 0$ in
    $\left( \mathcal{O}, \lbrace ~,~ \rbrace\right)$ and $\mathsf{z}_M$ satisfies the
    Maurer-Cartan equation in $\left( \mathsf{Graphs}_M, \delta, \lbrace ~,~ \rbrace _G \right)$.
    Applying $\varphi^m$ to the latter one results
    \begin{align*}
        0 &= \varphi^m \left( \delta \mathsf{z}_M + \frac{1}{2} \lbrace \mathsf{z}_M, \mathsf{z}_M \rbrace \right) \\
            &= \varphi^m \left( \left( \delta_{\text{split}} + \delta_{\text{pair}} \right) \mathsf{z}_M + \frac{1}{2} \lbrace \mathsf{z}_M, \mathsf{z}_M \rbrace _G  \right) \\
            &= \varphi^m \left( \delta_{\text{split}} \mathsf{z}_M \right) + \Delta (\varphi^m (\mathsf{z}_M))
            + \varphi^m \left( \frac{1}{2} \lbrace \mathsf{z}_M, \mathsf{z}_M \rbrace _G \right) \\
            &= \Delta (\mathsf{Z}_M^m) + \frac{1}{2} \lbrace \mathsf{Z}_M^m, \mathsf{Z}_M^m \rbrace _{\mathcal{S}}
    \end{align*}
    where $\varphi^m$ is compatible with the brackets and
    $\varphi^m \left( \delta_{\text{split}} \mathsf{z}_M \right)$ vanishes due to lemma (\ref{lmm:ds}):
    \begin{align*}
        \varphi^m \left( \delta_{\text{split}} \mathsf{z}_M \right) 
            = (-1) ^{\vert \mathsf{z}_M \vert + 1}  \mathsf{Z}_M ^{\lbrace m,m \rbrace} = 0
    \end{align*}
\end{proof}
\noindent
As always, a Maurer-Cartan element gives rise to a twisted differential
$\Delta^m = \Delta + \lbrace \mathsf{Z}_M^m,~  \rbrace _{\mathcal{S}}$.
With these modifications we can show analogous to proposition (\ref{prop:dcptl}) that
\begin{prop}
    The map $\varphi^m: \left(\mathcal{D} ^m, \partial ^m, \lbrace , \rbrace _{\mathcal{D}} \right) \rightarrow
    \left( \mathcal{S}, \Delta^m, \lbrace ,  \rbrace _{\mathcal{S}}   \right)  $ is a morphism of dg Lie $\mathsf{e}_n$-modules.
\end{prop}
\begin{proof}
    For $\Gamma \in \mathsf{Graphs}_M$ and $f_J \in \mathcal{O} ^{\otimes} $ we compute with the help of lemma (\ref{lmm:ds}) 
    \begin{align*}
        \varphi^m \big( \partial \left( \Gamma \otimes f_J \right) \big)
        &= \left( \delta_{\text{split}} + \delta_{\text{pair}} + \delta^{\mathsf{z}_M} \right) \Gamma \comp_m f_J
        + (-1) ^{\vert \Gamma \vert}  \Gamma \comp_m d^m f_J \\
        &= \left( \delta_{\text{pair}} + \delta^{\mathsf{z}_M} \right) \Gamma \comp_m f_J \\
        &= \left( \delta_{\text{pair}} \Gamma \right) \comp_m f_J + \lbrace \mathsf{z}_M, \Gamma \rbrace \comp_m f_J \\
        &= \Delta \left( \Gamma \comp_m f_J \right) + \lbrace \mathsf{Z}_M^m, \Gamma \comp_m f_J \rbrace \\
        &= \Delta^m \left( \Gamma \comp_m f_J \right)
    \end{align*}
\end{proof}
\noindent
\begin{proof}[Proof of Theorem \ref{thm:main2}]
We are now in the position to reduce this case to Theorem (\ref{thm:main1}).
Therefore we filter by the number of factors $\hbar$ in
$\mathcal{D}^m$ and $\mathcal{S}$ respectively. 
Only the differentials $d^m$ and $\lbrace \mathsf{Z}_M^m, ~ \rbrace$ raise this 
degree by one and therefore vanish in the associated graded.  All other differentials 
remain unchanged and we obtain as the associated graded 
complex the same complex $\mathcal{D}$ as before in subsection (\ref{sub:Filtrations}).
\end{proof}

%next line adds the Bibliography to the contents page
\addcontentsline{toc}{section}{Bibliography}
%uncomment next line to change bibliography name to references
\bibliography{main}
\bibliographystyle{ieeetr}  %use the plain bibliography style
\end{document}